\documentclass[12pt]{amsart}

\usepackage{latexsym, amssymb, amsmath,ulem,soul,esint}

\usepackage{amsfonts, graphicx}
\usepackage{graphicx,color}
\newcommand{\be}{\begin{equation}}
\newcommand{\ee}{\end{equation}}
\newcommand{\beq}{\begin{eqnarray}}
\newcommand{\eeq}{\end{eqnarray}}

\newtheorem{thm}{Theorem}[section]

\newtheorem{lma}{Lemma}[section]
\newtheorem{prop}{Proposition}[section]

\newtheorem*{ack*}{Acknowledgements}

\newtheorem{claim}{Claim}[section]
\theoremstyle{remark}
\newtheorem{rem}{Remark}[section]
\numberwithin{equation}{section}

\def\be{\begin{equation}}
\def\ee{\end{equation}}
\def\bee{\begin{equation*}}
\def\eee{\end{equation*}}

\def\lf{\left}
\def\ri{\right}

\newcommand{\Ric}{\mathrm{Ric}}

\def\Ric{\text{\rm Ric}}
\def\Rm{\text{\rm Rm}}

\def\p{\partial}

\def\heat{\lf(\frac{\p}{\p t}-\Delta_{g(t)}\ri)}

\def\e{\varepsilon}
\def\a{{\alpha}}

\def\R{\mathbb{R}}

\begin{document}

\title[]
{Scalar curvature lower bound under integral convergence}

\author{Yiqi Huang}
\address[Yiqi Huang]{Department of Mathematics, MIT, 77 Massachusetts Avenue, Cambridge, MA 02139-4307, USA}
 \email{yiqih777@mit.edu}

 \author{Man-Chun Lee}
\address[Man-Chun Lee]{Department of Mathematics, The Chinese University of Hong Kong, Shatin, Hong Kong, China}
\email{mclee@math.cuhk.edu.hk}


\renewcommand{\subjclassname}{
  \textup{2010} Mathematics Subject Classification}

\date{\today}

\begin{abstract}
In this work, we consider sequences of $C^2$ metrics which converges to a $C^2$ metric in $C^0$ sense. We show that if the scalar curvature of the sequence is almost non-negative in the integral sense, then the limiting metric has scalar curvature lower bound in point-wise sense.
\end{abstract}

\maketitle

\markboth{Scalar curvature under integral convergence}{}

\section{Introduction}

In view of compactness geometry, it is important to understand notions of curvature lower bounds in spaces with low regularity.  It is natural to ask if $C^0$ structure is sufficient to define the scalar curvature lower bound as opposed to the theory of Ricci curvature and sectional curvature.  In \cite{Gromov2014}, Gromov considered the stability problems and proved 
\begin{thm}[Gromov]
Let $M$ be a (possibly open) smooth manifold and $\kappa$ is a continuous function. Consider a sequence of $C^2$ Riemannian metrics $g_i$ on $M$ that converges to a $C^2$ Riemannian metric $g_0$ in $C^0_{loc}$. Assume that for all $i$, the scalar curvature of $g_i$ satisfies $\mathrm{R}(g_i)\geq \kappa$ on M. Then $\mathrm{R}(g_0)\geq \kappa$ on $M$.
\end{thm}

In \cite{Bamler2016}, Bamler gave an alternative proof of Gromov's result using Ricci flow and a uniqueness result of Koch and Lamm \cite{KochLamm2012}. In \cite{Burkhardt2019}, Burkhardt-Guim generalized Bamler's result to the case when $g_0$ is only $C^0$. In this case, She is able to give an alternative definition (the $\beta$-weak sense) of lower bound on scalar curvature of $C^0$ metric on closed manifolds. The preservation of scalar curvature's lower bound had played a important role in this works. More recently in \cite{JiangShengZhang}, Jiang, Sheng and Zhang consider the scalar curvature lower bounds in distributional sense for singular metrics in $W^{1,p}$ for some $n < p \leq \infty$. In particular, they prove that the weak notion of scalar curvature lower bound is preserved along Ricci flow. 

Motivated by the above mentioned works, it is natural to ask if the point-wise scalar curvature lower bound can be weakened further to integral forms under $C^0$ convergence in contrast with the work in \cite{JiangShengZhang}. To this end, in this work we generalize the Theorem in two directions.  We first consider the case when the scalar curvature lower bound of $g_i$ is almost $\kappa$ in $L^{n/2}_{loc}$ sense.  We show that if a sequence of $C^2$ metrics $g_i$ converges in $C^0_{loc}$ to some $C^2$ metric $g_0$, then $g_0$ has scalar curvature bounded from below by $\kappa$.

\begin{thm}\label{Main-Theorem-1}
Let $M^n,n\geq 3$ be a (possibly open) Riemannian manifold and $\kappa: M \rightarrow \mathbb{R}$ is a continuous function. Consider a sequence of $C^2$ Riemannian metrics $g_{i,0}$ on $M$ that converges to a $C^2$ Riemannian metric $g_0$ in the local $C^0$-sense on $M$ as $i\to+\infty$. Assume that for any ball $B_{g_0}(x, r)$ with $r\leq 1$, 
\begin{equation}\label{assumption-1}
    \int_{B_{g_0}(x, r)} (\mathrm{R}(g_i)-\kappa)_{\_}^{\frac{n}{2}} d\mu_{g_0} \rightarrow 0.
\end{equation}
Then $\mathrm{R}(g_0) \geq \kappa$ everywhere on $M$. 
\end{thm}

\begin{rem}
In view of the case where the negative part of scalar curvature drifts to infinity, the condition $\int_M (\mathrm{R}(g_{i,0})-\kappa)_{\_}^{\frac{n}{2}} d\mu_{g_i} \rightarrow 0$ is stronger than our assumption.
\end{rem}

We note that thanks to the $C^0$ convergence, it is flexible to interchange the measurement of volume and distance with respect to $g_{i,0}$ and $g_0$. The $L^{n/2}$ condition is in certain sense scaling invariant. 

If we weaken $L^{n/2}$ to $L^p$ with $p<n/2$, it is in general not strong enough to obtain estimates. To this end, we consider the case of some weighed $L^1$. We show that when $M$ is compact, a similar result is true.

\begin{thm}\label{Main-Theorem-2}
Let $M^n,n\geq 3$ be a compact Riemannian manifold with a $C^2$ metric $g_0$. Suppose there is a sequence of $C^2$ Riemannian metrics $g_{i,0}$ on $M$ that converges to  $g_0$ in the $C^0$-sense on $M$.  Suppose there is $\delta>0,\sigma\in \R$ such that for any $r\leq \mathrm{diam}(M,g_0)$ and $x\in M$,  
\begin{equation}\label{assumption-1}
    r^{(2-n-2\delta)}\int_{B_{g_0}(x, r)} \left( \mathrm{R}(g_{i,0})-\sigma\right)_{\_} d\mu_{g_0} \leq \e_i\rightarrow 0.
\end{equation}
Then $\mathrm{R}(g_0) \geq \sigma$ everywhere on $M$. 
\end{thm}

\begin{rem}
If the limit metric is only $C^0$, with the estimates in our proof and arguments in \cite{Burkhardt2019}, the scalar curvature of limit metric is greater than $\sigma$ in the $\beta$-weak sense.
\end{rem}

In particular, the curvature can be a-priori unbounded as $r\to 0$. The Theorem says that the $C^0$ convergence of metrics is sufficient to rule out the blow-up of scalar curvature if the rate is mild.

\begin{ack*}
The authors would like to thank Professor Richard Bamler for his interest in this note and helpful discussions.  The authors would like to thank the referee for useful comments.
\end{ack*}

\section{Ricci-Deturck flow and Ricci flow} \label{Pre-RF}

In this section, we will collect some useful information on Ricci-Deturck flows. Let $(M,h)$ be a complete Riemannian manifold. A $C^2$ family of metrics $\hat g(t)$ is said to be a Ricci-Deturck flow with background metric $h$ if it satisfies 
\begin{equation}
    \left\{
    \begin{array}{ll}
    \partial_t g_{ij}=-2R_{ij}+\nabla_i W_j +\nabla_j W_i;\\
    W^k=g^{pq}\left(\Gamma^k_{pq}-\Gamma^k(h)_{pq} \right);\\
    g(0)=g_0.
    \end{array}
    \right.
\end{equation}

In particular, the Ricci-Deturck flow is strictly parabolic. Moreover, it is equivalent to the Ricci flow in the following sense. Let $\Phi_t$ be a diffeomorphism given by 
\begin{equation}
    \left\{
    \begin{array}{ll}
    \partial_t \Phi_t(x)=-W\left(\Phi_t(x),t \right);\\
    \Phi_0(x)=x.
    \end{array}
    \right.
\end{equation}
Then $\hat g(t)=\Phi_t^*g(t)$ is a Ricci flow solution with $\hat g(0)=g(0)$.

The following result concerns the existence and uniqueness of the Ricci-Deturck flow on Euclidean spaces for metrics which are bilipschitz to Euclidean metric on $\mathbb{R}^n$. The following Lemma could be found in \cite[Lemma 2]{Bamler2016}, see also \cite{Simon2002,KochLamm2012} for the origin of this works.

\begin{lma}\label{Lemma}
There exist constants $\epsilon >0$, $A>0$ such that the following is true: Let $g_0$ be a $C^2$ metric on $\mathbb{R}^n$ that is $(1+\epsilon$)-bilipschitz to the standard Euclidean metric $g_{Eucl}$. Then there exists a solution $g(t),t\in [0,1]$ to the Ricci-DeTurck flow with $g(0)=g_0 $ such that following holds:
\begin{enumerate}
    \item[(i)] $g(t)$ is smooth for $t>0$ and $g(t)\to g_0$ in $C^2_{loc}$ as $t\to 0$;
    \item[(ii)] For any $t\geq 0$, the metric $g(t)$ is $(1.1)$-bilipschitz to $g_{Eucl}$. 
    \item[(iii)] For any $t>0$ and $k \leq 10$, we have 
    \begin{equation*}
        |D^k g(t)| < \frac{A}{t^{\frac{k}{2}}}
    \end{equation*}
    where $D$ denotes the Euclidean derivatives.
    \item[(iv)] If $g_{i}(t)$ and $g(t)$ satisfy the assumptions and if  $g_i(0)$ converges to $g(0)$ uniformly in the $C^0$-sense, then $g_i(t)$ converges to $g(t)$ uniformly in the $C^0$-sense on $\mathbb{R}^n \times [0, 1]$ and locally in the smooth sense on $\mathbb{R}^n \times (0, 1)$. 
\end{enumerate}
\end{lma}

In particular, (iv) implies that if the limit is a-priori $C^2$, the limiting Ricci-Deturck flow obtained from $g_i(t)$ will coincide with the Ricci flow starting from $g_0$.  Hence, by studying the uniform properties of $g_i(t)$, it reveals the information of $g(t)$ and hence $g_0$ by letting $t\to0$. 
\bigskip

For the later purpose, we would also like to collect some useful information of Green function. We start with its definition. Given a complete Ricci flow $\hat g(t)$ on $M$. For $x,y\in M$ and $0\leq s<t\leq T$,  we use $G(x,t;y,s)$ to denote the heat kernel corresponding to the backwards heat equation coupled with the Ricci flow. Namely, 
\begin{equation}
(\partial_s +\Delta_{y,s}) \, G(x,t;y,s)=0,\quad\text{and}\quad \lim_{s\to t^-} G(x,t;y,s)=\delta_x(y)
\end{equation}
for any fixed $(x,t)\in M\times [0,T]$.  And for all $(y,s)\in M\times [0,T]$, we have 
\begin{equation}
(\partial_t -\Delta_{x,t}-\mathrm{R}_{g(t)})\, G(x,t;y,s)=0,\quad\text{and}\quad \lim_{t\to s^+} G(x,t;y,s)=\delta_y(x).
\end{equation}

In a recent work of Bamler, Cabezas-Rivas and Wilking \cite{BamlerCabezasWilking2019}, it was found that the Green function above with respect to the operator $\partial_t-\Delta_t-\mathrm{R}_t$ satisfies an estimate under a scaling invariant condition. 
\begin{prop}[Proposition 3.1 in \cite{BamlerCabezasWilking2019}]\label{GreenFunctionEstimate}
For any $n,A>0$, there is $C(n,A)$ such that the following holds: Let $(M,\hat g(t)), t\in [0,T]$ be a complete Ricci flow satisfying 
$$|\Rm(x,t)|\leq At^{-1},\quad\text{and}\quad \mathrm{Vol}_{\hat g(t)} \left(B_{\hat g(t)}(x,\sqrt{t}) \right)\geq A^{-1} t^{n/2}$$
for all $(x,t)\in M\times (0,T]$. Then 
$$G(x,t;y,s)\leq \frac{C}{t^{n/2}}\exp\left(-\frac{d_s^2(x,y)}{Ct} \right).$$
\end{prop}

In view of localization,  it is sometimes slightly more convenient to use Ricci flow although the Ricci Deturck flows and Ricci flows are diffeomorphic to each other. We will interchange between them depending on the purpose.

\section{Local $L^{n/2}$-Estimates}

In this section, we will first prove a local persistence estimates of $L^{n/2}$ along the \textit{Ricci flows}. This will play an important role in the proof of Theorem~\ref{Main-Theorem-1}.

\begin{prop} \label{Main-Estimate} 
Suppose $(M^n,g(t)), n\geq 3$ is a complete Ricci flow satisfying 
$$|\Rm(x,t)|\leq At^{-1},\quad\text{and}\quad \mathrm{Vol}_{g(t)}\left(B_{g(t)}(x,1) \right)\leq A$$
for some $A>0$ on $M\times (0,T]$. For any $\sigma\in \R$, there is $C(n,A,\sigma),  S(n,A,\sigma)>0$ such that for all $t< \min\{S,T\}$, we have 
$$ \int_{B_{g(t)}\left(x_0,\frac14\right)} (\mathrm{R}_{g(t)}-\sigma)_{\_}^{\frac{n}{2}} \;d\mu_t \leq  e^{C_6\sqrt{t}}\cdot \int_{B_{g_0}\left(x_0,1\right)} (\mathrm{R}_{g_0}-\sigma)_{\_}^{\frac{n}{2}} \;d\mu_0+C_7 \sqrt{t}.$$

\end{prop}
\begin{proof}

In the following, we will use $C_i$ to denote constants depending only on $n,A,\sigma$. We will assume $S\leq 1$ so that we always assume working with $t\leq 1$.

Since $|\Rm( g(t))|\leq At^{-1}$ for some $A>0$ on $M\times (0,1]$,  we apply \cite[Lemma 8.3]{Perelman2002} to deduce that the evolving distance function $d_t(x,x_0)$ satisfies 
\begin{equation}
\heat d_t(x,x_0) \geq -C_0t^{-1/2}
\end{equation}
in the sense of barrier  whenever $d_t(x,x_0)\geq \sqrt{t}$.

Let $\phi: [0, \infty) \rightarrow [0, 1]$ be a smooth non-increasing function so that $\phi\equiv 1$ on $[0,\frac12]$, vanishes outside $[0,1]$ and satisfies $\phi''\geq -10^4\phi, |\phi'|\leq 10^4$. Then we define 
\begin{equation}
\Phi(x,t)=e^{-10^4mt} \phi^m\left({d_t(x,x_0)+2C_0\sqrt{t}} \right)
\end{equation}
where $m$ is a large integer to be chosen later. In this way,  direct computation show that function $\Phi$ is a cutoff function satisfying
\begin{equation}\label{cutoff-estimate}
\heat \Phi \leq 0
\end{equation}
in the sense of barrier and hence in the distribution sense, see \cite[Appendix A]{CarloGiovanniGennady}. Moreover, we have 
\begin{equation}\label{cutoff-estimate-2}
|\nabla \Phi|\leq 10^4 m\Phi^{1-\frac1m}.
\end{equation}
We will use $\Phi$ to localize the estimate with suitable choice of $m$.
\bigskip

Recall that the scalar curvature $\mathrm{R}$ satisfies 
\begin{equation}
 \heat \mathrm{R} = 2|\Ric|^2\geq \frac2n \mathrm{R}^2.
\end{equation}

Hence for $\sigma\in \mathbb{R}$, whenever $\mathrm{R}<\sigma$,  the function $\varphi = (\mathrm{R} - \sigma)\_$ satisfies 
\begin{equation}
\begin{split}
    \heat  \varphi &= \heat (-\mathrm{R})\\
    &\leq -\frac2n \mathrm{R}^2\\
    &\leq \frac2n (\mathrm{R}+\sigma) \varphi .
    \end{split}
\end{equation}

In particular, 
\begin{equation}\label{equ-varphi-n/2}
\begin{split}
\heat \varphi^{n/2}&\leq (\mathrm{R}+\sigma) \varphi^{n/2}-\left(\frac{n}{2}\right)\left(\frac{n}{2}-1 \right)\varphi^{n/2-2} |\nabla \varphi|^2
\end{split}
\end{equation}
in the sense of barrier whenever $\varphi>0$.  Therefore, the inequality holds in the distribution sense by \cite[Appendix A]{CarloGiovanniGennady}.

With this in hand, we consider the energy $E(t)=\int_M \varphi^{n/2} \Phi \; d\mu_{g(t)}$. By differentiating $E(t)$ with respect to $t$, we obtain from \eqref{equ-varphi-n/2} and \eqref{cutoff-estimate} that
\begin{equation}
\begin{split}
\frac{d}{dt} E(t)&=\int_M  \heat \left(\varphi^{n/2} \cdot \Phi \right) -\mathrm{R}  \cdot \varphi^{n/2}\Phi  \; d\mu_t \\
&\leq \int_M -\left(\frac{n}{2}\right)\left(\frac{n}{2}-1 \right)\varphi^{n/2-2} |\nabla \varphi|^2\Phi\;d\mu_t \\
&\quad +\int_M \sigma\varphi^{n/2} \Phi +n\varphi^{\frac{n}2-1}|\nabla \varphi| |\nabla \Phi| \; d\mu_t.
\end{split}
\end{equation}

Since $n\geq 3$, we apply Cauchy inequality to obtain
$$ n\varphi^{\frac{n}2-1} |\nabla \varphi||\nabla\Phi|\leq \left(\frac{n}{2}\right)\left(\frac{n-2}{2} \right)\varphi^{\frac{n}2-2} |\nabla \varphi|^2\Phi+C_n\varphi^\frac{n}{2}\frac{|\nabla\Phi|^2}{\Phi}$$
and therefore
\begin{equation}
\begin{split}
\frac{d}{dt} E(t)&\leq \sigma E(t)+C_1 \int_M \varphi^{n/2} \frac{|\nabla \Phi|^2}{\Phi} \;d\mu_t\\
&\leq \sigma E(t)+10^8C_1 m^2 \int_M  \varphi^{n/2} \Phi^{1-\frac2m} \; d\mu_t.
\end{split}
\end{equation}
where we have used \eqref{cutoff-estimate-2} on the last inequality.    From now on, we fix $m=2n$.  Recall that $|\Rm|\leq At^{-1}$ and hence $\varphi\leq C_2t^{-1}$. Therefore, 
\begin{equation}
\begin{split}
\frac{d}{dt} E(t)&\leq \sigma E(t)+C_3  \sup |\varphi|^{1/2} \cdot \int_M  \varphi^{\frac{n-1}{2}} \Phi^{1-\frac1n} \; d\mu_t\\
&\leq \sigma E(t)+\frac{C_4}{\sqrt{t}} E(t)^{1-\frac1n}\\
&\leq \frac{C_5}{\sqrt{t}} \left( E(t)+1\right).
\end{split}
\end{equation}
The last inequality follows from Young's inequality,  the fact that $\{x: \Phi\neq 0\} \subset B_{g(t)}(x_0,1)$, the volume estimate together with the Jensen inequality.

By integrating over $t\in [0,T]$,  we obtain 
\begin{equation}
\begin{split}
E(t)\leq  e^{C_6\sqrt{t}}E(0)+ \left(e^{C_6\sqrt{t}}-1\right).
\end{split}
\end{equation}
 
Since from the construction,  we have
\begin{equation}
B_{g(t)}\left(x_0,\frac14\right) \subset \{ x\in M: \Phi \equiv 1\}.
\end{equation} 
 
 Hence,  by shrinking $S$ we conclude that 
 \begin{equation}
 \int_{B_{g(t)}\left(x_0,\frac14\right)} (\mathrm{R}_{g(t)}-\sigma)_{\_}^{\frac{n}{2}} \;d\mu_t \leq  e^{C_6\sqrt{t}}\cdot \int_{B_{g_0}\left(x_0,1\right)} (\mathrm{R}_{g_0}-\sigma)_{\_}^{\frac{n}{2}} \;d\mu_0+C_7 \sqrt{t}.
 \end{equation}
 
 This completes the proof.
\end{proof}

\bigskip

Now we are ready to prove Theorem~\ref{Main-Theorem-1}. 
\begin{proof}[Proof of Theorem~\ref{Main-Theorem-1}]
We will follow the same spirit as in the work of Bamler  \cite{Bamler2016} but replacing the estimates by our new localization.

Suppose on the contrary, there is $x_0\in M$ such that $\mathrm{R}_{g_0}(x_0)< \kappa(x_0)$.

By scaling, we will assume $\Phi:U\subset M \to \mathbb{R}^n$ to be a chart centred at $x_0$ so that
\begin{enumerate}
\item[(a)] The metric $\Phi_{*}g$ is (1+$\epsilon$)-bilipschitz to $g_{Eucl}$ where the constant $\epsilon$ is obtained from Lemma~\ref{Lemma}.
    \item[(b)] $\Phi(B_{g_0}(x_0, \frac12)) \subset B_{euc}(0, 1) \subset \Phi(B_{g_0}(x_0, 2)) \subset B_{euc}(0, 3) \subset \Phi(U)$.
\end{enumerate}

For notational convenience, we will simply treat $\Phi$ to be identity locally.

\bigskip

We let $\sigma>0$ be such that $\mathrm{R}_{g_0}(x_0)\leq \kappa(x_0)-2\sigma$. By the continuity of $\kappa$, for $0<\e<<\sigma$, there exists $r_0>0$ so that for  $\kappa_0=\kappa(x_0)$,
\begin{equation}\label{cont-kappa}
\mathrm{R}_{g_0}(x)\leq \kappa_0-\sigma<\kappa_0 -\e<\kappa(z)\leq \kappa_0+\e
\end{equation}
for all $x,z\in B_{euc}(0,r_0)$. We may assume $r_0$ to be smaller than $\frac14$.

\bigskip

We choose a cutoff function $\phi$ on $\mathbb{R}^n$ which is identically 1 on $B_{euc}(0, 3)$ and supported in $U$. Then the metric $\tilde g_0 := \phi\cdot   g_0 + (1-\phi)g_{Eucl}$ satisfies:
\begin{enumerate}
    \item[(i)] The metric $\tilde  g_0$ is $C^2$ and is ($1+\epsilon$)-bilipschitz to $g_{Eucl}$.
    \item[(ii)] $\tilde g_0$ coincides with $ g_0$ on $B_{euc}(0,3)$.
\end{enumerate}

\bigskip

We also construct the perturbed approximating sequence $$\tilde g_{i,0}=\phi  g_i+(1-\phi) g_{euc}$$
on $\mathbb{R}^n$ so that $\tilde g_{i,0} \to \tilde g_0$ in $C^0(\mathbb{R}^n)$. By Lemma~\ref{Lemma}, for each $i\in \mathbb{N}$ there is a Ricci-Deturck flow $\tilde g_i(t),t\in [0,1]$ with $\tilde g_i(0)=\tilde g_{i,0}$ and 
\begin{equation}\label{RDF-estimate}
|D \tilde g_i(t)|^2 +|D^2 \tilde g_i(t)|\leq At^{-1}
\end{equation}
for some $A$ independent of $i$. 

Moreover, there is a Ricci-Deturck flow $\tilde g(t)$ starting from $\tilde g_0$ so that $\tilde g(t)$ is $C^2$ up to $t=0$ and $\tilde g_i(t)\to \tilde g(t)$ as $i\to +\infty$ in $C^\infty_{loc}$ for $t>0$ after passing to subsequence.  Thanks to the uniqueness, it suffices to control the scalar curvature of $\tilde g_i(t)$ uniformly.
\bigskip

For notational convenience, we will omit the index $i$ for the moment.  Recall that from the discussion in Section~\ref{Pre-RF}, $\hat g(t)=\Phi_t^* \tilde g(t)$ is a solution to the Ricci flow with $\hat g(0)=\tilde g(0)$ where 
\begin{equation}
\left\{
\begin{array}{ll}
\partial_t \Phi_t(x)=-W\left(\Phi_t(x),t \right);\\
\Phi_0(x)=x.
\end{array}
\right.
\end{equation}

Moreover,  since the Ricci-Deturck flow $\tilde g(t)$ satisfies $|D\tilde g(t)|^2+|D^2\tilde g(t)|\leq At^{-1}$ by passing \eqref{RDF-estimate} to limit, we see that $|\Rm(\tilde g(t))|\leq C_nAt^{-1}$ by comparing $\tilde g(t)$ with the Euclidean flat metric.  It also stays uniform equivalent to the Euclidean space.  Since the Ricci flow is diffeomorphic to the Ricci-Deturck flow,  $\hat g(t)$ satisfies the assumptions in Proposition~\ref{Main-Estimate}.

Noted that since $\tilde g(0)$ is $C^2$,  the Ricci flow solution is also $C^2$ up to $t=0$.  By applying Proposition~\ref{Main-Estimate} with scaling,  there is $S$ so that for $t\in [0,Sr_0^2]$,
 \begin{equation}\label{Equ-esti-RF-1}
 \begin{split}
&\quad \int_{B_{\hat g(t)}\left(0,\frac{r_0}4\right)} (\mathrm{R}_{\hat g(t)}-\kappa_0)_{\_}^{\frac{n}{2}} \;d\hat \mu_t \\
&\leq  e^{C_6r_0^{-1}\sqrt{t}}\cdot \int_{B_{euc}\left(0,2r_0\right)} (\mathrm{R}_{g_0}-\kappa_0)_{\_}^{\frac{n}{2}} \;d\mu_{euc}+C_7 r_0^{-1}\sqrt{t}
 \end{split}
 \end{equation}
 where we have used the fact that $\hat g(0)=\tilde g_0$ and the Ricci-Detruck flows are bi-lipschitz equivalent with Euclidean metric.

 Since the Ricci flow $\hat g(t)$ is diffeomorphic to the Ricci-Deturck flow $\tilde g(t)$ via $\Phi_t$, we might rewrite 
\begin{equation}
\begin{split}
 \int_{B_{\hat g(t)}\left(0,\frac{r_0}4\right)} (\mathrm{R}_{\hat g(t)}-\kappa_0)_{\_}^{\frac{n}{2}} \;d\hat \mu_t &= \int_{\Phi_t\left(B_{\hat g(t)}\left(0,\frac{r_0}4\right)\right)} (\mathrm{R}_{\tilde  g(t)}-\kappa_0)_{\_}^{\frac{n}{2}} \;d\tilde \mu_t .
\end{split}
\end{equation}
\bigskip

We now compare the integrating domain with geodesic balls. 
\begin{claim}
There exists $S>0$ independent of $i$ such that for all $t\in [0,Sr_0^2]$, 
$B_{\tilde g_0}\left(0,\frac{r_0}8 \right)\subset B_{\tilde g(t)}\left(\Phi_t(0),\frac{r_0}4 \right)=\Phi_t\left(B_{\hat g(t)}\left(0,\frac{r_0}4\right)\right)$ 
\end{claim}
\begin{proof}[Proof of Claim]
Since $\tilde g(t)$ is $1.1$-biLipschitz to $g_{euc}$, it is $1.2$-biLipschitz to $\tilde g_0$. Thus, if $d_{\tilde g_0}(x,0)<\frac{r_0}8$, 
\begin{equation}
\begin{split}
d_{\tilde g(t)}(\Phi_t(0),x)&\leq d_{\tilde g(t)}(\Phi_t(0),x_0)+d_{\tilde g(t)}(0,x)\\
&\leq (1.1) \cdot d_{euc}(\Phi_t(0),0)+(1.2) \cdot  d_{\tilde g_0}(0,x)\\
&\leq C_n\sqrt{t} +\frac3{20} r_0 
\end{split}
\end{equation}
where we have used the $|\partial_t \Phi|=|W|\leq At^{-1/2}$ from Lemma~\ref{Lemma}. This proved the claim by shrinking $S$ if necessary.
\end{proof}

Using the claim, \eqref{Equ-esti-RF-1} and Lemma~\ref{Lemma}, for all $i\to+\infty$, we have  for all $t\in (0,S]$,
\begin{equation}\label{estimate-limit}
\begin{split}
&\quad \int_{B_{ \tilde g_{i,0}}\left(0,\frac{r_0}8 \right)}  (\mathrm{R}_{\tilde g_i(t)}-\kappa_0)_{\_}^{\frac{n}{2}} \;d\tilde  \mu_{t} \\
&\leq  e^{C_6r_0^{-1}\sqrt{t}}\cdot \int_{B_{ euc}\left(0,2r_0\right)} (\mathrm{R}_{\tilde g_{i,0}}-\kappa_0)_{\_}^{\frac{n}{2}} \;d\mu_{euc}+C_7 r_0^{-1}\sqrt{t}.
\end{split}
\end{equation}

On the other hand, using the assumption and \eqref{cont-kappa}, we have for $i$ sufficiently large,
\begin{equation}
\begin{split}
&\quad \left[  \int_{B_{ euc}\left(0,2r_0\right)} (\mathrm{R}_{\tilde g_{i,0}}-\kappa_0)_{\_}^{\frac{n}{2}} \;d\mu_{euc}\right]^{2/n}\\
&\leq  \left[  \int_{B_{ euc}\left(0,2r_0\right)} (\mathrm{R}_{\tilde g_{i,0}}-\kappa)_{\_}^{\frac{n}{2}} \;d\mu_{euc}\right]^{2/n}+ \left[  \int_{B_{ euc}\left(0,2r_0\right)} |\kappa-\kappa_0|^{\frac{n}{2}} \;d\mu_{euc}\right]^{2/n}\\
&<C_n\e r_0^{2}.
\end{split}
\end{equation}

By letting $i\to+\infty$ on \eqref{estimate-limit} and followed by letting $t\to 0$, we conclude that 
\begin{equation}
\int_{B_{ \tilde g_{0}}\left(0,\frac{r_0}8 \right)}  (\mathrm{R}_{\tilde g_0}-\kappa_0)_{\_}^{\frac{n}{2}} \;d\tilde  \mu_{0} <C_n\e^{n/2} \cdot r_0^n.
\end{equation}

Since $\tilde g_0$ coincides with $g_0$ on $B_{ \tilde g_{0}}\left(0,\frac{r_0}8 \right)$, the left hand side can be estimated using \eqref{cont-kappa} to obtain 
\begin{equation}
\int_{B_{ \tilde g_{0}}\left(0,\frac{r_0}8 \right)}  (\mathrm{R}_{\tilde g_0}-\kappa_0)_{\_}^{\frac{n}{2}} \;d\tilde  \mu_{0}\geq  C_n^{-1}\sigma^{n/2} r_0^n
\end{equation}
which is impossible if $\e$ is chosen to be small relative to $\sigma$. This completes the proof.
\end{proof}

\section{$L^1$-Estimate}

In this section, we will consider the case when the scalar curvature's lower bound converges to some function in some weighed $L^1_{loc}$ sense.

We will start with the curvature estimate under \textit{Ricci flows}.
\begin{prop}\label{Improved-R-low}
Suppose $(M,g(t)),t\in [0,S]$ is a compact Ricci flow such that 
$$|\Rm(x,t)|\leq At^{-1},\quad\text{and}\quad \mathrm{Vol}_{g(t)} \left(B_{g(t)}(x,\sqrt{t}) \right)\geq A^{-1} t^{n/2}$$
for all $(x,t)\in M\times [0,S]$. If there are $\sigma\in \mathbb{R},\delta>0,\e\in (0,1)$ so that 
  $$  \sup_{\substack{r\in (0,D] \\ x \in M}}  r^{(2-n-2\delta)}\int_{B_{g_0}(x, r)} \left(\mathrm{R}(g_0)-\sigma \right)_{\_} d\mu_{g_0} < \epsilon,$$
  where $D=\mathrm{diam}(M,g_0)$,  then there is $\hat S(n,A,\sigma,\delta),L(n,A,\sigma,\delta)>0$ such that for all $t\in (0,S]\cap (0,\tilde S]$, 
$$ \left(\mathrm{R}(g(t))-\sigma \right)_{\_}\leq L \e t^{-1+\delta}.$$
\end{prop}
\begin{proof}
In what follows, we will use $C_i$ to denote constants depending only on $n,\sigma,\delta,A$. We may assume $\hat S\leq 1$ so that we always work with $t\leq 1$.

Recall that the function $\varphi=(\mathrm{R}-\sigma)_-$ satisfies 
\begin{equation}\label{evo-varphi-L1}
\begin{split}
\heat \varphi&\leq \frac2{n}(\mathrm{R}+\sigma) \varphi\\
&= \mathrm{R}\varphi +\left(\frac{2-n}{n}\mathrm{R}+\frac2n\sigma \right)\varphi.
\end{split}
\end{equation}

Let $\Lambda>0$ and $\a\in (0,1)$ be two numbers to be chosen. We let $T\in [0,S]$ be the maximal time such that for all $t\in (0,T)$, we have 
\begin{equation}
\varphi < \Lambda t^{-\a}
\end{equation}
and $\varphi(x_0,T)=\Lambda T^{-\a}$ for some $x_0\in M$.

Since $M$ is compact and the Ricci flow is smooth, we must have $T>0$. Our goal is to estimate $T$ from below if we choose $\Lambda,\a$ suitably.  Since $\varphi\leq \Lambda t^{-\a}$ on $M\times (0,T]$,   we deduce that 
\begin{equation}
\begin{split}
\frac{2-n}{n}\mathrm{R} +\frac2n \sigma&=\frac{2-n}{n}(\mathrm{R}-\sigma)+\frac{4-n}{n} \sigma\\
&\leq \varphi +\frac{4-n}{n} \sigma\\
&\leq \Lambda t^{-\a}+| \sigma|.
\end{split}
\end{equation}

Together with \eqref{evo-varphi-L1},   we obtain
\begin{equation}
\begin{split}
\heat  \tilde\varphi\leq \mathrm{R} \tilde\varphi
\end{split}
\end{equation}
in the sense of barrier where $\tilde \varphi=e^{-|\sigma|t-\frac{\Lambda}{1-\a} t^{1-\a}} \varphi $.


Hence, maximum principle implies 
\begin{equation}
\tilde\varphi(x,t)\leq \int_M G(x,t;y,0) \cdot \tilde\varphi(y,0) \; d\mu_0
\end{equation}
where $G(x,t;y,s)$ is the Green function of $\partial_t-\Delta_{g(t)} -\mathrm{R}$ associated to the Ricci flow $g(t)$, see Section~\ref{Pre-RF}. 

Therefore at $(x_0,T)$, we compute using Proposition~\ref{GreenFunctionEstimate} to yield
\begin{equation}
\begin{split}
&\quad \Lambda T^{-\a}\cdot \left(e^{-|\sigma|T-\frac{\Lambda}{1-\a} T^{1-\a}}  \right)\\
&= \tilde\varphi(x_0,T)\\
&\leq \int_M G(x_0,T;y,0) \cdot \tilde\varphi(y,0) \; d\mu_0\\
&\leq  \frac{C_1}{T^{n/2}} \sum_{k=0}^{\infty}\int_{B_{g_0}(x_0,(k+1)\sqrt{T})\setminus B_{g_0}(x_0,k\sqrt{T})} \varphi(y,0) exp(-\frac{d_0^2(x_0, y)}{C_1T}) d\mu_0 \\
&\leq  \frac{C_1}{T^{n/2}} \sum_{k=0}^\infty e^{-\frac{k^2}{C_1}} \int_{B_{g_0}(x_0,(k+1)\sqrt{T})} \varphi(y,0) \; d\mu_0\\
&\leq  \frac{C_1\e }{T^{1-\delta}} \sum_{k=0}^\infty e^{-\frac{k^2}{C_1}}(k+1)^{n-2+2\delta} \\
&\leq \frac{C_2\e}{T^{1-\delta}}.
\end{split}
\end{equation}

By choosing $\a=1-\delta$ and $\Lambda=e\cdot C_2\e$, we see that 
$$ 1\leq  |\sigma|T+\frac{e C_2\e}{1-\a} T^{1-\a} \leq C_3 T^{1-\a}$$
since we have assumed $T\leq 1$ and $0<\e<1$. In particular,  $1\geq T\geq C_4^{-1}$ for some $C_4>0$. In other word, we have shown that for all $t\in [0, C_4^{-1}]\cap [0,S]$,  
\begin{equation}
(\mathrm{R}(g(t))-\sigma)_- \leq  C_5 \e t^{\delta-1}.
\end{equation}

This completes the proof.
\end{proof}

Now we are ready to prove Theorem~\ref{Main-Theorem-2}.
\begin{proof}[Proof of Theorem~\ref{Main-Theorem-2}]
The proof is a slight modification of the proof of Theorem \ref{Main-Theorem-1} but we will use a recent result of Burkhardt-Guim \cite{Burkhardt2019} to simplify some argument.  By the main result of Simon \cite{Simon2002}, we can construct a Ricci Deturck flow $g(t),t\in [0,S]$ with $g(0)=g_0$ by defining $g(t)=\lim_{i\to +\infty} g_i(t)$ where $g_i(t)$ is a Ricci-Deturck flow with initial metric $g_i(t)=g_{i,0}$ using some fixed background metric $h$. The convergence is $C^0(M\times [0,S])\cap C^\infty_{loc}(M\times (0,S))$.  Moreover, both $g_i(t)$  and $g(t)$ is bi-Lipschitz to some fixed metric on $M$ and satisfy $|\Rm|\leq At^{-1}$ for some $A>0$.

Since $g_i(t)$ is diffeomorphic to the Ricci flow and is uniformly bi-Lipschitz, we can apply Proposition~\ref{Improved-R-low} to the Ricci flow to deduce that (after pulling back) for all $t\in (0,S]\cap (0,\tilde S]$, 
\begin{equation}
 \left(\mathrm{R}(g_i(t))-\sigma \right)_{\_}\leq L t^{-1+\delta} \cdot \e_i
\end{equation}

By letting $i\to +\infty$ on $(0,S]\cap (0,\tilde S]$, we conclude that the limiting Ricci-Deturck flow satisfies $\mathrm{R}(g(t)) \geq \sigma$ on $(0,S]\cap (0,\tilde S]$.  By \cite[Theorem 5.4]{Burkhardt2019} (more precisely its proof),  $g(t)$ is diffeomorphic to the $C^2$ Ricci flow $h(t)$ starting from $g_0$.  Hence, $\mathrm{R}(g_0)\geq \sigma$ on $M$ by letting $t\to 0$ on $h(t)$.
\end{proof}

\end{document}